\documentclass{amsart}  
\usepackage{latexsym,amssymb,palatino,eulervm,color,graphicx,amsmath,amsthm}
\usepackage{amsfonts}
\usepackage{amsthm}
\usepackage{amsmath}
\usepackage{amsfonts}
\usepackage{latexsym}
\usepackage{amssymb}
\usepackage{amscd}
\usepackage[latin1]{inputenc}
\usepackage{verbatim}

\newtheorem{definition}{Definition}[section]
\newtheorem{theorem}[definition]{Theorem}
\newtheorem{lemma}[definition]{Lemma}
\newtheorem{corollary}[definition]{Corollary}
\newtheorem{proposition}[definition]{Proposition}
\theoremstyle{definition}

\newcommand\style{\mathcal }          

\newcommand{\B}{\style{B}}
\newcommand{\M}{\style{M}}

\newcommand\A{{\style A}}
\renewcommand{\H}{\style{H}}
\newcommand{\K}{\style K}

\newcommand\fn{{\mathbb F}_n}         

\newcommand\ftwo{{\mathbb F}_2}
\newcommand{\lgG}{\rm{G}}        

\newcommand\osi{{\style I}}

\newcommand\osq{{\style Q}}
\newcommand\osr{{\style R}}
\newcommand\oss{{\style S}}
\newcommand\ost{{\style T}}

\newcommand\omin{\otimes_{\rm min}}

\newcommand\cstar{{\rm C}^*}                              
\newcommand\cstare{{\rm C}_{\rm e}^*}              
 
\def\mytimeA#1
 {%
   {%
      \count20=#1 %
      \count22=\count20
      \divide \count20 by 60
      \count21=\count20
      \multiply\count21 by -60
      \advance\count22 by \count21
      \ifnum\count20<12
           \def\tapm{ a.m.}%
      \else
           \def\tapm{ p.m.}%
      \fi
      \ifnum\count20=0
           \count20=12
      \fi
      \the\count20:%
      \ifnum\count22<10 0\fi
      \the\count22
   }%
 }

\title[Purity of Embeddings of Operator Systems]{Purity of the Embeddings of Operator Systems into their C$^*$- and Injective Envelopes}

\author{Douglas Farenick}
\author{Ryan Tessier}

\address{Department of Mathematics and Statistics, University of Regina,
Regina, Saskatchewan S4S 0A2, Canada}


\thanks{This work is supported in part by NSERC (Canada).}

\begin{document}
 
\begin{abstract} We study the issue of issue of purity (as a completely positive linear map)
for identity maps on operators systems and for their completely isometric embeddings into their
C$^*$-envelopes and injective envelopes. Our most general result states that
the canonical embedding of an operator system $\osr$ into its injective envelope
$\mbox{\rm I}(\osr)$ is pure if and only if the C$^*$-envelope $\cstare(\osr)$ of $\osr$ is a prime C$^*$-algebra. To prove this, we also
show that the identity map on any AW$^*$-factor is a pure completely positive linear map.

For embeddings of operator systems $\osr$ into their C$^*$-envelopes, the issue of purity is seemingly harder to describe 
in full generality, and so we focus here on operator systems arising from the generators of discrete groups. Two such 
operator systems of interest are denoted by $\oss_n$ and $\mbox{\rm NC}(n)$. The former corresponds to the generators of the free group
$\fn$, while the latter corresponds to the generators of the group $\mathbb Z_2*\cdots*\mathbb Z_2$, the free product of $n$ copies of $\mathbb Z_2$.
The operator systems $\oss_n$ and $\mbox{\rm NC}(n)$ are of interest in operator theory for their connections to the weak expectation property and C$^*$-nuclearity,
and for their universal properties. Specifically,
$\oss_n$ is the universal operator system for arbitrary $n$-tuples of contractions acting on a Hilbert space and
$\mbox{\rm NC}(n)$ is the universal operator system for $n$ selfadjoint contractions. We show that the embedding
of $\oss_n$ into $\cstare(\oss)$ is pure for all $n\ge2$ and the embedding of
$\mbox{\rm NC}(n)$ into $\cstare\left(\mbox{\rm NC}(n)\right)$
is pure for every $n\ge3$. 

The question of purity of the identity is quite subtle for operator system that are not C$^*$-algebras, and we have results only for the
operator systems $\oss_n$ and $\mbox{\rm NC}(n)$.

Lastly, a previously unrecorded feature of pure completely positive linear maps is presented: every pure
completely positive linear map on an operator system $\osr$ into an injective von Neumann algebra $\mathcal M$ has a pure
completely positive extension to any operator system $\ost$ that contains $\osr$ as an operator subsystem, thereby generalising 
a result of Arveson for the injective type I factor $\B(\H)$.
\end{abstract}

\maketitle

\section{Introduction}

A face $\mathcal F$ in a proper convex cone $\mathcal C$ is a half-line face \cite[p.~182]{Rockafellar-book}
if there exists an element
$\phi\in\mathcal C$ such that
\[
\mathcal F= \left\{t\phi\,|\,t\in\mathbb R,\,t\geq0\right\}.
\]
Half-line faces of convex cones are the analogues, for cones, of the notion of extreme points of convex sets. Under 
topological conditions such as closedness, convex combinations of elements taken from various 
half-line faces of a convex cone $\mathcal C$ completely determine the cone $\mathcal C$ \cite[Theorem 18.5]{Rockafellar-book}. 
In particular, any generator $\phi$ of a half-line face of a convex cone $\mathcal C$
is an extreme point of any convex subset $\mathcal E\subseteq\mathcal C$ that contains $\phi$.

The purpose of this paper is to study such faces in the case where the cone $\mathcal C$
is the cone $\mathcal C\mathcal P(\osr,\ost)$
of completely positive linear maps from an operator system $\osr$ into an operator system $\ost$. The generators $\phi$
of such half-line faces of $\mathcal C\mathcal P(\osr,\ost)$ are said to be \emph{pure completely positive linear maps}.
Expressed differently,
a completely positive linear map $\phi:\osr\rightarrow\ost$ between operator systems $\osr$ and $\ost$ is pure
in the cone $\mathcal C\mathcal P(\osr,\ost)$
if, for any completely positive linear maps $\vartheta,\omega:\osr\rightarrow\ost$
such that $\vartheta+\omega=\phi$, there necessarily exists a scalar $s\in[0,1]$ such that $\vartheta=s\phi$ and $\omega=(1-s)\phi$.

Henceforth, $\mathcal C\mathcal P(\osr,\ost)$ will be denoted by $\mathcal C\mathcal P(\osr)$ when the domain $\osr$
and co-domain $\ost$ are equal
(as operator systems).

If $\osr$ is a unital C$^*$-algebra $\A$ and if $\ost=\B(\H)$, the von Neumann algebra of bounded linear operators on a
Hilbert space $\H$, 
then there is a very satisfactory and readily applicable criterion discovered by
Arveson: namely, $\phi$ is a pure element of 
$\mathcal C\mathcal P\left(\A,\B(\H)\right)$ if and only if, for any minimal Stinespring decomposition 
 $\phi=w^*\pi w$ of $\phi$, the representation $\pi$ is irreducible \cite[Corollary 1.4.3]{arveson1969}.
Replacing the C$^*$-algebra $\A$ by an operator system $\osr$ is somewhat more problematic, 
but there is, nevertheless, a geometric criterion for a unital completely positive (ucp) linear map $\phi:\osr\rightarrow\M_n(\mathbb C)$,
where $\M_n(\mathbb C)$ is the C$^*$-algebra of $n\times n$ complex matrices, to be a pure element of 
$\mathcal C\mathcal P\left(\osr,\M_n(\mathbb C)\right)$: namely, 
$\phi$ is pure if and only if $\phi$ is a matrix extreme point in the 
compact free convex set $\mbox{\sf S}(\osr)$ of all matrix-valued ucp maps on $\osr$ \cite{farenick2000}. Consequently,
the results of \cite{arveson1969} and \cite{farenick2000} indicate that
arbitrary elements of $\mathcal C\mathcal P\left(\A,\B(\H)\right)$ and
$\mathcal C\mathcal P\left(\osr,\M_n(\mathbb C)\right)$
can be viewed as operator (or matrix) convex combinations of pure completely positive linear maps. In this sense, then,
pure completely positive linear maps determine all completely positive linear maps.

The most general case occurs when both $\osr$ and $\ost$ are arbitrary operator systems; however, in such cases 
the absence of any structure theory (such as a Stinespring decomposition) makes the determination of pure elements
of $\mathcal C\mathcal P(\osr,\ost)$ very difficult.

Arveson's criterion demonstrates that the identity map $\iota:\B(\H)\rightarrow\B(\H)$ is a pure element of
$\mathcal C\mathcal P\left(\B(\H)\right)$. It is, therefore, natural to ask: for which operator systems $\osr$ is
the identity map $\iota:\osr\rightarrow\osr$ a pure element of the cone $\mathcal C\mathcal P\left(\osr\right)$?
A closely related question involves embeddings: if $\osr\subset\ost$ is an inclusion of operator systems, then is the
canonical inclusion map $\tilde\iota:\osr\rightarrow\ost$ a pure element of 
$\mathcal C\mathcal P\left(\osr,\ost\right)$? In particular, for this second question, is the inclusion of an
operator system $\osr$ into its C$^*$-envelope $\cstare(\osr)$ pure in the cone 
$\mathcal C\mathcal P\left(\osr,\cstare(\osr)\right)$?

Thus, in this paper we are concerned with the following two questions. 
\begin{enumerate}
\item[{(Q1)}] For which operators systems $\osr$ is the identity map $\iota:\osr\rightarrow\osr$ pure in the cone $\mathcal C\mathcal P\left(\osr \right)$?
\item[{(Q2)}] For which operators systems $\osr$ is the embedding $\iota_e:\osr\rightarrow\cstare(\osr)$ pure in the cone $\mathcal C\mathcal P\left(\osr,\cstare(\osr)\right)$?
\end{enumerate}

An operator system $\osi$ is  injective if, for every operator system $\osr$ and every operator system $\ost$ containing $\osr$ as an operator subsystem,
each completely positive linear map $\phi:\osr\rightarrow\osi$ has an extension to a completely positive linear map 
$\Phi:\ost\rightarrow\osi$.
Arveson's Hahn-Banach Extension Theorem \cite[Theorem 1.2.3]{arveson1969} states
that $\B(\H)$ is an injective operator system for every Hilbert space $\H$. 
Hamana \cite{hamana1979b} provided crucial additional information about inclusions of operator systems into injective operator systems, showing that every operator system $\osr$ is an operator subsystem of some minimal injective operator system $\mbox{\rm I}(\osr)$,
which is called the injective envelope of $\osr$. 
If we denote the canonical inclusion of $\osr$ into $\mbox{\rm I}(\osr)$ by $\iota_{\rm ie}$,
then this map is a unital completely positive order embedding of the operator system $\osr$ into the operator system
$\mbox{\rm I}(\osr)$. The final question addressed in this paper is:

\begin{enumerate}
\item[{(Q3)}] For which operator systems $\osr$
is the embedding $\iota_{\rm ie}:\osr\rightarrow\mbox{\rm I}(\osr)$ pure in the cone $\mathcal C\mathcal P\left(\osr,\mbox{\rm I}(\osr)\right)$?
\end{enumerate}

In this paper we provide a complete answer to question (Q3); we also answer question
(Q2) for operator systems that arise from generators of
discrete groups. Question (Q1), however, is difficult to answer in a generic manner. 
Therefore, we only address question (Q1) for two classes of finite-dimensional
operator systems, denoted by $\oss_n$ and $\mbox{\rm NC}(n)$. These two classes of operator systems
are relevant for their universal properties, for their encoding of quantum correlations, and
for their role in operator-algebraic questions concerning
the weak expectation property and C$^*$-nuclearity 
\cite{farenick--kavruk--paulsen--todorov2014,farenick--kavruk--paulsen--todorov2018,farenick--paulsen2012,fritz2012,kavruk2014}.

\section{Operator Systems and Pure Completely Positive Linear Maps}

Our general reference for operator system theory are the books of Paulsen \cite{Paulsen-book} and Effros and Ruan \cite{Effros--Ruan-book}.  

\subsection{Operator systems and their C$^*$- and injective envelopes} If $\osr$ is a complex $*$-vector space, then the space $\M_n(\osr)$ of
$n\times n$ matrices over $\osr$ is also a complex $*$-vector space in which the adjoint of an
$n\times n$ matrix $[x_{ij}]_{i,j=1}^n$ of elements $x_{ij}\in\osr$ is defined by $([x_{ij}]_{i,j=1}^n)^*=[x_{ji}^*]_{i,j=1}^n$.
A \emph{matrix ordering} of a complex $*$-vector space $\osr$ is a family $\left\{\mathcal C_n\right\}_{n\in\mathbb N}$ of subsets 
$\mathcal C_n$ of the real vector spaces 
$\left(\M_n(\osr)\right)_{\rm sa}$ of selfadjoint matrices over $\osr$ such that, for all $n,m\in\mathbb N$, 
(i) $\mathcal C_n$ is a convex cone, (ii) $\mathcal C_n\cap\left(-\mathcal C_n\right)=\{0\}$, and 
(iii) $\alpha^*x\alpha\in\mathcal C_m$ for all $x\in\mathcal C_n$ and all complex $n\times m$ matrices $\alpha$.

An element $e_\osr\in\mathbb C_1$ is an \emph{Archimedean order unit} for a matrix ordering $\left\{\mathcal C_n\right\}_{n\in\mathbb N}$ of a complex $*$-vector space $\osr$
if, for every $n\in\mathbb N$, the condition $e_\osr^{[n]}+\varepsilon Q\in\mathcal C_n$ holds for every real $\varepsilon>0$ only for $Q\in\mathcal C_n$. Here,
$e_\osr^{[n]}=e_\osr\oplus\cdots\oplus e_\osr$, the $n$-fold direct sum of $e_\osr$. Note that if 
there exists an Archimedean order unit for a matrix ordering of $\osr$, then there are infinitely
many choices for this order unit. 

Formally, an \emph{operator system} is a triple $\left(\osr, \left\{\mathcal C_n\right\}_{n\in\mathbb N}, e_\osr\right)$ consisting of a complex 
$*$-vector space $\osr$, a matrix ordering $\left\{\mathcal C_n\right\}_{n\in\mathbb N}$ of $\osr$, and a distinguished element $e_\osr$ that serves
as an Archimedean order unit for the matrix ordering $\left\{\mathcal C_n\right\}_{n\in\mathbb N}$ of $\osr$.
Unless it is necessary to make explicit reference to the matrix ordering $ \left\{\mathcal C_n\right\}_{n\in\mathbb N}$ and/or the order unit $e_\osr$,
the triple $\left(\osr, \left\{\mathcal C_n\right\}_{n\in\mathbb N}, e_\osr\right)$ will be denoted simply by $\osr$. The matrix cones $\mathcal C_n$
of an operator system $\osr$ are generally denoted by $\M_n(\osr)_+$.

Every unital C$^*$-algebra $\A$ 
is an operator system, where the matrix cones are given by $\mathcal C_n=\M_n(\A)_+$, the 
cone positive elements
of the C$^*$-algebra $\M_n(\A)$ and the Archimedean order unit is the multiplicative identity $1$ of $\A$, 
which is the canonical choice of Archimedean order unit for this
matrix ordering of $\A$.

If $\osr$ and $\ost$ are operator systems, then a linear map $\phi:\osr\rightarrow\ost$ is completely positive if
$\phi^{(n)}:\M_n(\osr)\rightarrow\M_n(\ost)$ maps $\M_n(\osr)_+$ into $\M_n(\ost)_+$, for every $n\in\mathbb N$, where
$\phi^{(n)}$ is defined by 
$\phi^{(n)}\left(\left[x_{ij}\right]_{i,j=1}^n\right)=\left[\phi(x_{ij})\right ]_{i,j=1}^n$. If, in addition, $\phi(e_\osr)=e_\ost$, then
$\phi$ is said to be a unital completely positive linear map, or a ucp map.
A linear isomorphism $\phi:\osr\rightarrow\ost$ in which both $\phi$ and $\phi^{-1}$ are completely positive is called a 
\emph{complete order isomorphism}. An one-to-one completely positive
linear map $\phi:\osr\rightarrow\ost$ in which $\phi(\osr)$ is an operator subsystem
of $\ost$ is called 
a \emph{complete order embedding} of $\osr$ into $\ost$ if 
 $\phi$ is a complete order isomorphism when considered as a map of the operator system
 $\osr$ onto the operator system $\phi(\osr)$.
 
The notation $\osr\simeq\ost$ indicates the existence of a unital complete order isomorphism
$\phi:\osr\rightarrow\ost$, although we shall also have need of complete order isomorphisms that are not unital.


The Choi--Effros Embedding Theorem \cite{choi--effros1977} states that every operator system $\osr$ is unitally 
completely order isomorphic to an
operator subsystem of $\B(\H)$ for some Hilbert space $\H$. Therefore, every operator system $\osr$ is capable of generating
a C$^*$-algebra. Of all such possibilities, our interest is with the minimal one, which is called the
C$^*$-envelope of $\osr$ and whose existence was first established by Hamana \cite{hamana1979b}.

\begin{definition} A \emph{C$^*$-envelope} of an operator system $\osr$ is a pair $(\A, \iota_e)$ consisting of
\begin{enumerate}
\item a unital C$^*$-algebra $\A$, and 
\item a unital complete order embedding $\iota_e:\osr\rightarrow\A$ such that $\iota_e(\osr)$ generates the C$^*$-algebra $\A$  
\end{enumerate}
such that, for every unital complete order embedding $\kappa:\osr\rightarrow\B$ of $\osr$ into a unital C$^*$-algebra $\B$ for which $\kappa(\osr)$ generates $\B$,
there exists a unital $*$-homomorphism $\pi:\B\rightarrow\A$ with $\pi\circ\kappa=\iota_e$.
\end{definition}

\begin{theorem}\label{hamana}{\rm (Hamana, \cite{hamana1979b})} Every operator system $\osr$ admits a C$^*$-envelope $(\A,\iota_e)$. Furthermore, 
if $(\tilde\A, \tilde\iota)$ is any other C$^*$-envelope of $\osr$, then there exists a unital $*$-isomorphism $\varrho:\tilde\A\rightarrow\A$ such that
$\varrho\circ\tilde\iota=\iota_e$.
\end{theorem}

Because, by Theorem \ref{hamana}, the C$^*$-envelope of an operator system is unique up to isomorphism, we shall use the notation $\cstare(\osr)$ to denote
a (the) C$^*$-envelope of $\osr$ and $\iota_e$ to denote a (the) unital complete order embedding of $\osr$ into $\cstare(\osr)$.

Recall that an operator system $\osi$ is \emph{injective} if, for every operator system $\osr$ and every operator system $\ost$ that contains $\osr$ as an
operator subsystem, every completely positive linear maps $\phi:\osr\rightarrow\osi$ has an extension $\Phi:\ost\rightarrow\osi$ such that $\Phi$ is completely
positive.

\begin{definition} An \emph{injective envelope} of an operator system $\osr$ is a pair $(\osi, \iota_{\rm ie})$ consisting of
\begin{enumerate}
\item an injective operator system $\osi$, and 
\item a unital complete order embedding $\iota_{\rm ie}:\osr\rightarrow\osi$ 
\end{enumerate}
such that, for every inclusion $\iota_{\rm ie}(\osr)\subseteq \osq\subseteq\osi$ as operator subsystems in which $\osq$ is injective, 
then necessarily $\osq=\osi$.
\end{definition}

\begin{theorem}\label{hamana op sys 1}{\rm (Hamana, \cite{hamana1979b})} Every operator system $\osr$ admits an injective $(\osi,\iota_{\rm ie})$. Furthermore,
if $(\tilde\osi, \tilde\iota)$ is any other injective envelope of $\osr$, then there exists a unital complete order isomorphism $\phi:\tilde\osi\rightarrow\osi$ such that
$\phi\circ\tilde\iota=\iota_{\rm ie}$.
\end{theorem}

In light of the uniqueness of the injective envelope up to complete order isomorphism, an injective envelope of $\osr$ is denoted by $\mbox{\rm I}(\osr)$.

By the Choi-Effros Embedding Theorem \cite{choi--effros1977}, every operator system $\osr$ is completely order isomorphic
to an operator subsystem of $\B(\H)$, for some Hilbert space $\H$. Thus, every operator system is an operator subsystem of
an injective operator system. However, Hamana \cite{hamana1979b}
proved that if $\osr$ is an operator subsystem of $\B(\H)$, then there exists
a unital completely positive idempotent map $\mathcal E:\B(\H)\rightarrow\B(\H)$ 
such that the range of $\mathcal E$ is an (the) injective of $\osr$. This operator system $\mbox{\rm I}(\osr)$ is 
unitally completely order isomorphic to an injective C$^*$-algebra $\B$ via the Choi-Effros product: specifically,
$\B$ is given by the operator system $\mbox{\rm I}(\osr)$ whereby the product $x\circ y$ of $x,y\in \mbox{\rm I}(\osr)$
is defined to be $x\circ y=\mathcal E(xy)$, with $xy$ denoting the product of $x$ and $y$ in $\B(\H)$.
Within $\B$, the unital C$^*$-algebra generated by $\osr$ is isomorphic to the C$^*$-envelope $\cstare(\osr)$  of $\osr$. 
These results are summarised in the theorem below, along with a crucial property known as rigidity.

\begin{theorem}\label{hamana op sys}{\rm (Hamana, \cite{hamana1979b})} For every operator system $\osr$, the injective
envelope $\mbox{\rm I}(\osr)$ of $\osr$ has the following properties:
\begin{enumerate}
\item {\rm (C$^*$-envelope)} there is an injective C$^*$-algebra $\B$ and a unital complete order isomorphism $\phi:\mbox{\rm I}(\osr)\rightarrow\B$
such that, if $\A$ denotes the C$^*$-subalgebra of $\B$ generated by $\phi\left( \iota_{\rm ie}(\osr)\right)$, then
$(\A, \phi\circ\iota_{\rm ie})$ is a C$^*$-envelope of $\osr$;
\item {\rm (Rigidity)} if $\omega:\mbox{\rm I}(\osr)\rightarrow\mbox{\rm I}(\osr)$ is a completely positive linear map for which 
$\omega\circ\iota_{\rm ie}=\iota_{\rm ie}$, then $\omega$ is the identity map on $\mbox{\rm I}(\osr)$.
\end{enumerate}
\end{theorem}

Theorem \ref{hamana op sys} indicates that the injective envelope of $\osr$, when viewed as a unital injective C$^*$-algebra,
contains a copy of the C$^*$-envelope of $\osr$ as a unital C$^*$-subalgebra.

\subsection{Dual operator systems, minimal tensor products, and entanglement}

As every operator system is a normed vector space \cite[p.~179]{choi--effros1977}, the dual space $\osr^d$ of an operator system $\osr$
is a Banach space. A matrix ordering of $\osr^d$ occurs when we declare an $n\times n$ matrix
$G=[\gamma_{ij}]_{i,j=1}^n$ of linear functionals on $\osr$ to be positive 
if the linear function $\hat G:\osr\rightarrow\M_n(\mathbb C)$ defined by
\[
\hat G(x)=\left[ \gamma_{ij}(x)\right]_{i,j=1}^n, \mbox{ for }x\in\osr,
\]
is completely positive \cite[Lemma 4.2]{choi--effros1977}. 
While it is not true that this matrix ordering admits an Archimedean order unit for every operator system, 
the duals of finite-dimensional operator systems do possess an 
Archimedean order unit for this matrix ordering
\cite[Corollary 4.5]{choi--effros1977}.
Indeed, every faithful state $\delta$ on $\osr$ is an 
Archimedean order unit \cite[Lemma 2.5]{kavruk2014} for the matrix ordering of $\osr^d$. Thus, 
the choice of Archimedean order unit $\delta$ for the dual $\osr^d$
of a finite-dimensional operator system $\osr$ is not canonical.

With regards to questions of purity, the use of an operator system dual can be useful in light of the following straightforward result.

\begin{proposition}\label{pure dual} Suppose that $\osr$ and $\ost$ are finite-dimensional operator systems and that $\phi:\osr\rightarrow\ost$
is a completely positive linear map. Then:
\begin{enumerate}
\item the linear adjoint $\phi^d:\ost^d\rightarrow\osr^d$ is completely positive, and
\item $\phi$ is pure in $\mathcal C\mathcal P(\osr, \ost)$ if and only if 
$\phi^d$ is pure in $\mathcal C\mathcal P(\ost^d, \osr^d)$.
\end{enumerate}
\end{proposition}


The algebraic tensor product $\osr\otimes\ost$ of operator systems $\osr$ and $\ost$ is a complex $*$-vector space,
and there are many possible matrix orderings that are induced by the matrix orderings of $\osr$ and $\ost$ that give 
$\osr\otimes\ost$ the structure of an operator system \cite{kavruk--paulsen--todorov--tomforde2011}. 
In the case where $\osr=\M_k(\mathbb C)$ and $\ost=\M_m(\mathbb C)$, there is a unique operator system
tensor product structure on $\osr\otimes\ost$ and it is the one induced by considering $M_n(\M_k(\mathbb C)\otimes \M_m(\mathbb C)$
as a unital C$^*$-algebra.

\begin{definition} {\rm (\cite[\S4]{kavruk--paulsen--todorov--tomforde2011})} 
The \emph{minimal operator system tensor product} $\osr\omin\ost$
of operator systems $\osr$ and $\ost$  is the operator system whose matrix ordering 
is defined so that a matrix $X=[x_{ij}]_{i,j=1}^n\in\M_n(\osr\otimes\ost)$ is positive if
\[
\left[(\phi\otimes\psi)(x_{ij})\right]_{i,j=1}^n \in \left( M_n(\M_k(\mathbb C)\otimes \M_m(\mathbb C))\right)_+,
\]
for all unital completely positive linear maps $\phi:\osr\rightarrow\M_k(\mathbb C)$ and
$\psi:\ost\rightarrow\M_m(\mathbb C)$ and for all $k,m\in\mathbb N$,
and whose canonical Archimedean order unit is given by $e_\osr\otimes e_\ost$.
\end{definition}

The minimal operator system tensor product $\osr\omin\ost$ of 
operator systems $\osr$ and $\ost$ may be realised by representing $\osr$ and $\ost$ as operator subsystems of $\B(\H)$ and $\B(\K)$, respectively, and
then endowing the vector space $\osr\otimes\ost$ of operators on the Hilbert space $\H\otimes\K$ with the operator system structure induced by the operator system structure of $\B(\H\otimes\K)$
\cite[Theorem 4.4]{kavruk--paulsen--todorov--tomforde2011}.


If $\mathcal V$ is a finite-dimensional vector space, then the tensor product $\mathcal V\otimes \mathcal V^d$ of $\mathcal V$ with its dual $\mathcal V^d$
is linearly isomorphic to $\mathcal L(\mathcal V)$, the vector space of linear transformations on $\mathcal V$. If we apply this linear isomorphism to a finite-dimensional operator system $\osr$ and its operator system dual $\osr^d$,
then the cone $\mathcal C\mathcal P(\osr)$ in $\mathcal L(\osr)$ of completely positive linear maps on $\osr$ determines a cone in $\osr\otimes\osr^d$. This cone is in fact the positive cone of $\osr\omin\osr^d$, where
$\omin$ denotes the minimal operator system tensor product \cite[\S4]{kavruk--paulsen--todorov--tomforde2011}, \cite[Lemma 8.5]{kavruk--paulsen--todorov--tomforde2013}.

The canonical linear isomorphism between $\osr\otimes\osr^d$ and $\mathcal L(\osr)$ is the one that maps elementary tensors $x\otimes\psi\in\osr\otimes\osr^d$ to rank-1 linear transformations $r\mapsto\psi(r)x$, for $r\in\osr$. Let
$\Gamma:\mathcal L(\osr)\rightarrow \osr\otimes\osr^d$ be the inverse of this canonical linear isomorphism. Thus, 
\[
\Gamma\left(\mathcal C\mathcal P(\osr)\right)=\left(\osr\omin\osr^d\right)_+,
\]
the cone of positive elements of the operator system $\osr\omin\osr^d$. Thus, it is clear that $\phi\in\mathcal C\mathcal P(\osr)$ is pure if and only if $\Gamma(\phi)$ is pure (i.e., only if $\Gamma(\phi)$ generates a half-line face of the cone 
$\left(\osr\omin\osr^d\right)_+$).

\begin{definition}{\rm (\cite{kavruk2015})} If $\osr$ is a finite-dimensional operator system, then an 
element $\xi\in\osr\otimes\osr^d$ is \emph{maximally entangled} if there exist bases $\{x_0,\dots,x_m\}$ and $\{\delta_0,\dots,\delta_m\}$ of $\osr$ and $\osr^d$, respectively, such that
\begin{enumerate}
\item $\{x_0,\dots,x_m\}$ and $\{\delta_0,\dots,\delta_m\}$ are dual bases (i.e., $\delta_i(x_i)=1$ and $\delta_i(x_j)=0$ if $j\not=i$, for all $i$ and $j$), 
\item $x_0=e_\osr$ and $\delta_0=e_{\osr^d}$, and
\item $\xi=\displaystyle\sum_{j=0}^m x_j\otimes \delta_j$.
\end{enumerate}
\end{definition}

\begin{proposition}\label{max ent} Let $\iota:\osr\rightarrow\osr$ be the identity map of a finite-dimensional operator system $\osr$. Then $\Gamma(\iota)$ is the unique maximally entangled 
element of $\osr\otimes\osr^d$.
\end{proposition} 

\begin{proof} Set $x_0=e_\osr$ and let $\delta_0$ be a faithful state on $\osr$; thus, $\delta_0$ is a positive linear functional, $\delta_0(x_0)=1$, and $\delta_0$ serves as an Archimedean order unit $e_{\osr^d}$ for
$\osr^d$. Select a basis $\{x_1,\dots,x_m\}$ of $\ker\delta_0$
and a dual basis $\{\delta_1,\dots,\delta_m\}$ for the vector space $\ker\delta_0$; this dual basis can be realised by linear functionals on $\osr$ such that 
$\{x_0,x_1,\dots,x_m\}$ and $\{\delta_0,\delta_1,\dots,\delta_m\}$ are dual bases for $\osr$ and $\osr^d$. Set $\xi=\displaystyle\sum_{j=0}^m x_j\otimes \delta_j$, which by definition
is maximally entangled. Under the canonical isomorphism
$\osr\otimes\osr^d\rightarrow \mathcal L(\osr)$, each $x_i\otimes\delta_i$ maps to an operator on $\osr$ that annihilates every $x_j$ with $j\not=i$ and fixes $x_i$. That is, $\Gamma^{-1}(\xi)=\iota$,
making 
$\Gamma(\iota)=\xi$. The uniqueness of $\xi$ as a maximally entangled element is a consequence of the fact that if $\xi'$ were any other maximally entangled element, then $\Gamma^{-1}(\xi')$
must be $\iota$, implying that $\xi'=\xi$.
\end{proof}

\begin{corollary}\label{max ent pure} The identity map $\iota:\osr\rightarrow\osr$ of a finite-dimensional operator system $\osr$ is pure if and only if the maximally entangled element $\Gamma(\iota)\in\left(\osr\omin\osr^d\right)_+$ is pure.
\end{corollary}

\subsection{Boundary representations} 
Boundary representations of the C$^*$-algebra generated by an operator system have an important role in 
the proving our results on the purity of identity mappings.

\begin{definition}
If $\A$ is a unital C$^*$-algebra generated by an operator system $\osr$, then a representation
$\pi:\A\rightarrow\B(\K)$ of $\A$ on some Hilbert space $\K$ is a \emph{boundary representation for $\osr$} if (i) $\pi$ is irreducible and
(ii) $\pi$ is the unique ucp extension to $\A$ of the completely positive linear map $\pi_{\vert\osr}:\osr\rightarrow\B(\K)$.
\end{definition}

The first tool we shall use is the following one. 

\begin{proposition}\label{br->p} If $\A$ is a unital C$^*$-algebra generated by an operator system $\osr$, and if 
$\pi:\A\rightarrow\B(\K)$ is a boundary representation of $\osr$, then $\pi_{\vert\osr}$ is a pure element of 
$\mathcal C\mathcal P(\osr,\B(\K))$.
\end{proposition} 

\begin{proof} Let $\phi=\pi_{\vert\osr}$, where $\pi:\A\rightarrow\B(\K)$ is a boundary representation of $\osr$. Suppose that
$\vartheta,\omega:\osr\rightarrow\B(\K)$ are completely positive linear maps
such that $\vartheta+\omega=\phi$. By the Arveson Hahn-Banach Extension Theorem
\cite[Theorem 1.2.3]{arveson1969},
there are completely positive linear extensions $\Theta$ and
$\Omega$ of $\vartheta$ and $\omega$ from $\osr$ to $\A$. Thus, $\Theta+\Omega$ is a completely positive extension of
$\phi$; hence, $\pi=\Theta+\Omega$. Because $\pi$ is irreducible, the Radon-Nikod\'ym
Theorem for completely positive linear maps \cite[Theorem 1.4.2]{arveson1969} 
implies that $\Theta=s\pi$ and $\Omega=t\pi$,
for some $s,t\in[0,1]$. Hence, $\vartheta=s\phi$ and $\omega=t\phi$, proving that $\phi$ is pure.
\end{proof} 

 The crucial link between boundary representations and C$^*$-envelopes, given by the following
 Choquet-type theorem, is the second tool from the theory of boundary representations 
 that is required to prove
 our purity results.
 
 \begin{theorem}\label{bndry thm}{\rm (Arverson-Davidson-Kennedy \cite{arveson2008,davidson--kennedy2015})}
Suppose that $\osr$ is an operator subsystem of a unital C$^*$-algebra $\A$ and that $\A=\cstar(\osr)$. Let
$\mathfrak S$ denote the ideal of $\cstar(\osr)$ given by
 \[
 \mathfrak S=\{a\in\cstar(\osr)\,|\,\pi(a)=0 \mbox{ for every boundary representation }\pi\mbox{ of }\osr\},
 \]
 and let $q:\cstar(\osr)\rightarrow\cstar(\osr)\rightarrow\cstar(\osr)/\mathfrak S$ denote the canonical quotient homomorphism. Then
 $\left(\cstar(\osr)/\mathfrak S, q_{\vert \osr} \right)$ is a C$^*$-envelope of $\osr$. That is,
 $\cstare(\osr)=\cstar(\osr)/\mathfrak S$ and the ucp map $\iota_e=q_{\vert\osr}$ is a complete order embedding
 of $\osr$ into its C$^*$-envelope $\cstar(\osr)/\mathfrak S$. Furthermore, for any matrix $X\in\M_n(\osr)$, the norm of $X$ is given
 by
 \[
 \|X\|=\max\left\{\|\pi^{(n)}(X)\|\,|\,\pi\mbox{ is a boundary representation of }\osr\right\},
 \]
 where $\pi^{(n)}$ denotes the unital $*$-homomorphism on $\M_n\left(\cstar(\osr)\right)$ that maps every matrix
 $\left[a_{ij}\right]_{i,j=1}^n$ of elements of $\cstar(\osr)$ to $\left[\pi(a_{ij})\right]_{i,j=1}^n$.
 \end{theorem}
 
 \subsection{Noncommutativity and the purity of ucp maps}
 The following simple observation will be useful in cases where the codomain of a unital completely positive linear map is an algebra.
 
 \begin{proposition}\label{trivial centre} If a unital C$^*$-algebra $\A$ has nontrivial centre and if 
 $\phi:\osr\rightarrow\A$
 a unital completely positive linear map, the $\phi$ is not a pure element of the cone $\mathcal C\mathcal P(\osr,\A)$.
 \end{proposition}
 
 \begin{proof} By the hypothesis that the centre $\mbox{\rm Z}(\A)$ of $\A$ is nontrivial, there exists a nonscalar positive element $a\in 
 \mbox{\rm Z}(\A)$ of norm $\|a\|=1$. If $\phi:\osr\rightarrow\A$ is a unital completely positive linear map, then define 
 $\vartheta,\omega:\osr\rightarrow\A$ by $\vartheta(x)=a^{1/2}\phi(x)a^{1/2}$ and $\omega(x)=(1-a)^{1/2}\phi(x)(1-a)^{1/2}$, for $x\in\osr$. Observe that, for every $x\in\osr$,
\[
\vartheta(x)+\omega(x)=a^{1/2}xa^{1/2}+(1-a)^{1/2}\phi(x)(1-a)^{1/2}=\phi(x)\left(a+(1-a)\right)=\phi(x).
\]
However, as $a$ is nonscalar, $a=\vartheta(e_\osr)\not=\lambda\phi(e_\osr)$ for every $\lambda\geq0$. Hence, $\phi$ is not pure.
 \end{proof} 
 
\subsection{Extension of pure completely positive linear maps}

In addition to proving that the von Neumann algebra $\B(\H)$ is injective in his seminal paper \cite{arveson1969}, Arveson proved  
that if $\osr $ is an operator subsystem of an operator system $\ost$, then a pure completely positive linear map $\phi:\osr\rightarrow\B(\H)$
extends to a pure completely positive linear map $\Phi:\ost\rightarrow\B(\H)$. The following result shows that pure extensions
occur for pure maps into arbitrary injective von Neumann.

\begin{proposition}[Pure Extensions]\label{hb-thm} 
If $\osr\subseteq\ost$
is an inclusion of operator systems, and if $\phi:\osr\rightarrow\M$ is a pure completely positive linear
map into an injective factor $\M$, then
then $\phi$ extends to a pure completely positive linear map $\Phi:\ost\rightarrow\M$.
\end{proposition}

\begin{proof} Let $\delta=\|\phi\|$ and consider the set $\mathcal C\mathcal P_\delta(\osr,\B(\H))$ of 
all completely positive linear maps $\psi:\osr\rightarrow\M$ of norm $\|\psi\|\leq \delta$. In the point-ultraweak topology, 
$\mathcal C\mathcal P_\delta(\osr,\M)$ is a compact space \cite[Theorem 7.4]{Paulsen-book}. Consider the subset
$\mathcal C$ of $\mathcal C\mathcal P_\delta(\ost,\M)$ consisting of all completely positive linear maps $\Phi:\ost\rightarrow\M$
such that $\Phi_{\vert\osr}=\phi$. By the injectivity of $\M$, the set $\mathcal C$ is nonempty. It is also plainly convex. We now show
that $\mathcal C$ is compact in $\mathcal C\mathcal P_\delta(\ost,\M)$.

Let $\{\phi\}_{\alpha}$ be a net in $\mathcal C$. Because the norm of any completely positive map on an operator system is achieved at the order unit of the
operator system, we have that $\|\Phi\|=\delta$, for every $\Phi\in\mathcal C$; hence, 
$\mathcal C\subset \mathcal C\mathcal P_\delta(\ost,\M)$ and
$\{\phi\}_{\alpha}$ is a net in $\mathcal C\mathcal P_\delta(\ost,\M)$. By the compactness of $\mathcal C\mathcal P_\delta(\ost,\M)$, there exists a subnet 
$\{\phi_{\alpha_i}\}_i$ of $\{\phi\}_{\alpha}$ and a $\phi\in \mathcal C\mathcal P_\delta(\ost,\M)$ such that $\phi(x)$ is the limit, in the ultraweak topology of $\B(\H)$, of
the net $\{\phi_{\alpha_i}(x)\}_i$ of operators $\phi_{\alpha_i}(x)\in\M$. Because $\M$ is closed in the ultraweak topology of $\M$, the operator $\phi(x)$ must
belong to $\M$ for every $x$. Hence, the limiting map $\phi$ is an element of $\mathcal C$. Thus, because every net in $\mathcal C$ admits a convergent subnet, 
$\mathcal C$ is compact in the subspace topology of $\mathcal C\mathcal P_\delta(\ost,\M)$.

The point-ultraweak topology is a weak*-topology, and so the Krein-Milman Theorem applies to the compact convex set $\mathcal C$, which yields an
extreme point $\Phi$ of $\mathcal C$. 
Suppose that $\Phi=\Theta+\Omega$, for some nonzero
completely positive linear maps $\Psi,\Omega:\ost\rightarrow\M$. Thus, if
$\vartheta=\Theta_{\vert\osr}$ and $\omega=\Omega_{\vert\osr}$, then $\phi=\vartheta+\omega$; hence, by the purity of $\phi$, there are nonzero
$s,t\in[0,1]$ such that $\vartheta=s\phi$ and $\omega=t\phi$. Therefore, the completely positive linear maps 
$\frac{1}{s}\Theta$ and $\frac{1}{t}\Omega$ are completely positive extensions of $\phi$, making them elements of $\mathcal C$, 
and these maps satisfy
$s(\frac{1}{s}\Theta)+t(\frac{1}{t}\Omega)=\Phi$. Furthermore, as $\phi(x)=s\phi(x)+t\phi(x)=(s+t)\phi(x)$ for every $x\in\osr$, we
have that $s+t=1$. Hence, $\Phi$ is a convex combination of $\frac{1}{s}\Theta$ and $\frac{1}{t}\Omega$. Since $\Phi$ is an extreme point
of $\mathcal C$, we obtain $\Phi=\frac{1}{s}\Theta=\frac{1}{t}\Omega$, which proves that the extension $\Phi$ of $\phi$
is pure.
 \end{proof}
 
 A consequence of the proof of Proposition \ref{hb-thm} is the following result.
 
 \begin{proposition}
If $\osr\subseteq\ost$
is an inclusion of operator systems and if $\phi:\osr\rightarrow\osi$ is a pure completely positive linear
map of $\osr$ into an injective operator system $\osi$, then every extreme point $\Phi$ of the convex set of extensions of $\phi$ to $\ost$
is a pure completely positive linear map $\Phi:\ost\rightarrow\osi$.
\end{proposition}

\subsection{Operator systems from discrete groups}

Suppose that $\mathfrak u$ is a generating set for a discrete group $\lgG$. Considered as a subset of the (full) group C$^*$-algebra,
$\cstar(\lgG)$, each element of $\mathfrak u$ is unitary. Define $\oss(\mathfrak u)\subset\cstar(\lgG)$ by 
\[
\oss(\mathfrak u)=\mbox{Span}\left\{1,u,u^*\,|\,u\in\mathfrak u\right\},
\]
where $1$ denotes the multiplicative identity of the C$^*$-algebra $\cstar(\lgG)$.
Thus, $\oss(\mathfrak u)$ is an operator system in $\cstar(\lgG)$. 

\begin{definition} Let $u_1,u_2, u_3, \dots\,$ denote 
generators of the free group $\mathbb F_\infty$. 
For each $n\in\mathbb N$, let 
\[
\mathcal S_n=\mbox{\rm Span}\{1,u_k,u_k^*\,|\,k=1,\dots,n\}\subset\cstar(\mathbb F_n),
\]
where $\fn$ is the free group generated by $u_1,\dots,u_n$. The operator system $\oss_n$ is called the
\emph{operator system of the free group $\fn$}.
\end{definition}

Up to unital complete order isomorphism, the definition of $\oss_n$ is independent of the choice of generators of $\mathbb F_n$.
Moreover, it is clear that $\cstar(\oss_n)=\cstar(\fn)$, for every $n\in\mathbb N$.

\begin{definition}\label{nc defn} For each $n\in\mathbb N$ with $n\ge2$, let $v_1,\dots,v_n$ be generators of the group $*_1^n\mathbb Z_2$,
the free product of $n$ copies of $\mathbb Z_2$, and define 
\[
\mbox{\rm NC}(n)=\mbox{\rm Span}\{1,v_1,\dots,v_n\}\subset\cstar(*_1^n\mathbb Z_2).
\]
(Note that in $\cstar(*_1^n\mathbb Z_2)$ each $v_j$ is a selfadjoint unitary (i.e., $v_j^2=1$).) The operator system 
$\mbox{\rm NC}(n)$ is called the \emph{operator system of the noncommutative $n$-cube}.
\end{definition}
 
 The condition $n\ge2$ in Definition \ref{nc defn} is present only to justify the use of the term ``noncommutative cube". One can
 also define $\mbox{\rm NC}(1)$ in the obvious manner.
 
As with free groups, 
up to unital complete order isomorphism the definition of $\mbox{\rm NC}(n)$ is independent of the choice of generators of 
$*_1^n\mathbb Z_2$, and $\cstar(\mbox{\rm NC}(n))=\cstar(*_1^n\mathbb Z_2)$.  
 
The following result describes the C$^*$-envelope of operator systems arising from discrete groups.

\begin{theorem}\label{env gp}{\rm (\cite[Proposition 3.2]{farenick--kavruk--paulsen--todorov2014})}
If $\mathfrak u$ is a generating set for a discrete group $\lgG$, then
$\left(\cstar(\mbox{\rm G}), \tilde\iota\right)$ is a C$^*$-envelope for $\mathcal S(\mathfrak u)$, where
$\tilde\iota:\oss(\mathfrak u)\rightarrow \cstar(\lgG)$ is the canonical inclusion.
\end{theorem}

\begin{corollary} $\cstare(\oss_n)=\cstar(\fn)$ and $\cstare\left(\mbox{\rm NC}(n)\right)=\cstar(*_1^n \mathbb Z_2)$.
\end{corollary}

It is somewhat remarkable that the operator systems $\oss_n$ and
$\mbox{\rm NC}(n)$ are operator system quotients of operator systems of matrices \cite{farenick--kavruk--paulsen--todorov2014,farenick--paulsen2012}. As a consequence, the 
dual operator systems $\oss_n^d$ and
$\mbox{\rm NC}(n)^d$ can be realised by operator systems of matrices, which is crucial for our analysis of the purity of identity maps on
the operator systems $\oss_n$ and $\mbox{\rm NC}(n)$.

\begin{theorem}\label{duals} For every choice of order unit for each of the operator system duals $\oss_n^d$ and
$\mbox{\rm NC}(n)^d$,
\begin{enumerate}
\item {\rm (\cite[Theorem 4.5]{farenick--paulsen2012})} 
the operator system dual $\oss_n^d$ of $\oss_n$ is completely order isomorphic to
the operator system $\mathcal X_n$ of $2n\times 2n$ matrices defined by
\[
\mathcal X_n=\left\{\bigoplus_{k=1}^n\left[\begin{array}{cc} \alpha & \beta_k \\ \gamma_k & \alpha\end{array}\right]\,|\,
\alpha,\beta_k,\gamma_k\in\mathbb C,\,k=1,\dots,n\right\},
\]
and
\item {\rm (\cite[Proposition 6.1]{farenick--kavruk--paulsen--todorov2014})} 
the operator system dual $\mbox{\rm NC}(n)^d$ of $\mbox{\rm NC}(n)$ 
is completely order isomorphic to
the operator system $\mathcal Y_n$ of $2n\times 2n$ matrices defined by
\[
\mathcal Y_n=\left\{\bigoplus_{k=1}^n\left[\begin{array}{cc} \alpha & \beta_k \\ \beta_k & \alpha\end{array}\right]\,|\,
\alpha,\beta_k \in\mathbb C,\,k=1,\dots,n\right\}.
\]
\end{enumerate}
\end{theorem}

To conclude, we make explicit note of the earlier-mentioned universal properties of $\oss_n$ and $\mbox{\rm NC}(n)$.

\begin{theorem}\label{univ prop} Let $y_1,\dots,y_n\in\B(\H)$ be arbitrary contractions, for $n\in\mathbb N$, and let $h_1,\dots,h_m\in\B(\H)$
be arbitrary selfadjoint contractions, for $m\in\mathbb N$ with $m\geq2$. There exist unital completely positive linear maps
$\phi:\oss_n\rightarrow\B(\H)$ and $\psi:\mbox{\rm NC}(m)\rightarrow\B(\H)$ such that
\[
\phi(u_j)=y_j, \mbox{ for every }j, \mbox{ and } \psi(v_k)=h_k, \mbox{ for every }k.
\]
\end{theorem}

Theorem \ref{univ prop} is proved in \cite[Proposition 9.7]{kavruk--paulsen--todorov--tomforde2013} and
\cite[Proposition 6.5]{farenick--kavruk--paulsen--todorov2014}.

\section{Embedding Operator Systems into their Injective Envelopes}

The main result of this section, Theorem \ref{pure ie},
is a determination of the purity of the embedding
of an arbitrary operator system $\osr$ into its injective envelope by way of the C$^*$-envelope of $\osr$, thereby answering
question (Q3).

If $\B$ is an injective C$^*$-algebra, then $\B$ is a monotone complete C$^*$-algebra and, hence, an AW$^*$-algebra
\cite{Saito--Wright-book}. By the seminal work of Hamana \cite{hamana1981,Saito--Wright-book}, 
every unital C$^*$-algebra $\A$ has a regular monotone completion
$\overline A$ such that $\overline\A$ is a unital C$^*$-subalgebra of its injective envelope $\mbox{\rm I}(\A)$
(when $\mbox{\rm I}(\A)$ is considered in its guise as an injective C$^*$-algebra). Moreover, for any operator system $\osr$,
the following system of unital complete order embeddings holds:
\begin{equation}\label{hamana e}
\osr\subseteq\cstare(\osr)\subseteq  \overline{\cstare(\osr)} \subseteq \mbox{\rm I}(\osr).
\end{equation}


\begin{lemma}\label{factor lemma} If $\B$ is an AW$^*$-factor, then the identity map $\iota:\B\rightarrow\B$ is a pure element of
$\mathcal C\mathcal P(\B)$.
\end{lemma}

\begin{proof} Every AW$^*$-factor is a primitive C$^*$-algebra \cite{saito--wright2006}; hence, without loss of generality we may assume that
$\B$ is an irreducible C$^*$-algebra of operators acting on some Hilbert space $\H$. The identity map
$\iota:\B\rightarrow\B$  is, therefore, an irreducible representation of
$\B$ on $\H$; thus, by Arveson's criterion \cite[Corollary 1.4.3]{arveson1969}, $\iota$ is a pure element of $\mathcal C\mathcal P\left(\B,\B(\H)\right)$.

Suppose now that $\vartheta,\omega\in\mathcal C\mathcal P(\B)$ are such that $\vartheta+\omega=\iota$. 
Considering $\vartheta$ and $\omega$ as elements of $\mathcal C\mathcal P\left(\B,\B(\H)\right)$, the purity of $\iota$ in 
$\mathcal C\mathcal P\left(\B,\B(\H)\right)$ and
the equation
$\vartheta+\omega=\iota$ imply that $\vartheta=s\iota$ and $\omega=(1-s)\iota$ for some $s\in[0,1]$. Hence, it is also true that
$\iota:\B\rightarrow\B$ is a pure element of
$\mathcal C\mathcal P(\B)$.
\end{proof}

\begin{theorem}\label{pure ie} The following statements are equivalent for the canonical unital complete order embedding  
$\iota_{\rm ie}:\osr\rightarrow \mbox{\rm I}(\osr)$:
\begin{enumerate}
\item $\iota_{\rm ie}$ is pure in the cone $\mathcal C\mathcal P\left(\osr, \mbox{\rm I}(\osr)\right)$;
\item the C$^*$-algebra $\cstare(\osr)$ is prime.
\end{enumerate}
\end{theorem}

\begin{proof} To prove that (1) implies (2), we shall prove that $\iota_{\rm ie}$ is not pure if $\cstare(\osr)$ is not prime. To this end, let $\A=\cstare(\osr)$,
which we assume to be nonprime. The regular monotone completion of any nonprime C$^*$-algebra has nontrivial centre \cite[Theorem 7.1]{hamana1981}. 
Thus, the ucp map
$\iota_{\rm ie}:\osr\rightarrow\overline A$ is not pure in the cone $\mathcal C\mathcal P(\osr,\overline A)$,
by Proposition \ref{trivial centre}. Hence, via the system of inclusions (\ref{hamana e}) above,
$\iota_{\rm ie}$ is not pure in the cone $\mathcal C\mathcal P\left(\osr, \mbox{\rm I}(\osr)\right)$.

Conversely, to prove that (2) implies (1), assume that $\cstare(\osr)$ is a prime C$^*$-algebra. Thus, $\mbox{\rm I}\left(\cstare(\osr)\right)$ is (unitally completely
order isomorphic to)
an injective AW$^*$-factor \cite[Theorem 7.1]{hamana1981}. Suppose that $\vartheta,\omega:\osr\rightarrow\mbox{\rm I}(\osr)$ are completely positive
linear maps such that $\vartheta+\omega=\iota_{\rm ie}$. By the injectivity of $\mbox{\rm I}(\osr)$, there exists completely positive extensions $\Theta$
and $\Omega$ of $\vartheta$ and $\omega$, respectively, from $\osr$ to $\mbox{\rm I}(\osr)$. Thus, $\Theta+\Omega$ is a ucp map
on $\mbox{\rm I}(\osr)$ for which $(\Theta+\Omega)[x]=x$, for every $x\in\osr$. By the rigidity property of the injective envelope \cite[Lemma 3.6]{hamana1979b}, 
it is necessarily 
the case that $(\Theta+\Omega)[z]=z$, for every $z\in \mbox{\rm I}(\osr)$. However, as $\mbox{\rm I}(\osr)$ is an AW$^*$-factor, 
the identity map on $\mbox{\rm I}(\osr)$ is pure, by Lemma \ref{factor lemma}; thus, $\Theta$ and $\Omega$ are scalar multiples of the identity map
on $\mbox{\rm I}(\osr)$, implying that $\vartheta$ and $\omega$ are scalar multiples of $\iota_{\rm ie}$. Hence, 
$\iota_{\rm ie}$ is pure in the cone $\mathcal C\mathcal P\left(\osr, \mbox{\rm I}(\osr)\right)$.
 \end{proof}

\begin{corollary} The embedding of $\oss_n$ into its injective envelope is pure for every $n\geq 2$.
\end{corollary}

\begin{proof} The C$^*$-envelope of $\oss_n$ is $\cstar(\fn)$, which is primitive (and, hence prime) for every $n\geq2$
\cite{choi1980}. Thus, Theorem \ref{pure ie} yields the purity of the embedding $\iota_{\rm ie}$.
\end{proof}

\section{Embedding Discrete-Group Operator Systems into their C$^*$-envelopes}

This section answers question (Q2) for
operator systems from discrete groups.

\begin{lemma}\label{kav} Suppose 
that $\osr$ is an operator 
subsystem of a unital C$^*$-algebra $\A$ for which $\A=\cstar(\osr)$, and that $\osr$
contains a set $\mathfrak u$ of unitary elements of $\A$ that generate $\A$ as a C$^*$-algebra. Suppose that
$\B$ is a unital C$^*$-subalgebra of $\B(\H)$, for some Hilbert space $\H$. If
$\phi:\osr\rightarrow\B$ is a ucp map for which $\phi(u)$ is unitary, for every $u\in\mathfrak u$, and
if the commutant of $\{\phi(u)\,|\,u\in\mathfrak u\}$ in $\B(\H)$ is $1$-dimensional,
then
$\phi$ is a pure element of $\mathcal C\mathcal P(\osr,\B)$.
\end{lemma} 

\begin{proof} The hypothesis that $\phi(u)$ is unitary, for every $u\in\mathfrak u$, implies that 
$\phi$ admits a unique ucp extension $\pi:\cstar(\osr)\rightarrow\B$ and that the extension
$\pi$ is a homomorphism \cite[Lemma 9.3]{kavruk--paulsen--todorov--tomforde2013}. 
Because the commutant of $\{\phi(u)\,|\,u\in\mathfrak u\}$ in $\B(\H)$ is $1$-dimensional, the C$^*$-algebra
$\pi\left(\cstar(\osr)\right)$ is irreducible; in other words, $\pi$ is an irreducible representation
of $\cstar(\osr)$. 

Suppose that $\vartheta$ and $\omega$ are completely positive linear maps $\osr\rightarrow\B$ such that
$\phi=\vartheta+\omega$. In considering $\vartheta$ and $\omega$ as elements of 
$\mathcal C\mathcal P\left(\osr,\B(\H)\right)$,
Arveson's Extension Theorem yields completely positive extensions $\Theta$ and $\Omega$
of $\vartheta$ and $\omega$ that map $\cstar(\osr)$ into $\B(\H)$. Hence, $\Theta+\Omega$ is one
completely positive extension of $\phi$, considering $\phi$ as an element of $\mathcal C\mathcal P\left(\osr,\B(\H)\right)$.
However,
because $\phi(u)$ is unitary in $\B(\H)$ for every $u\in\mathfrak u$,
$\phi$ has unique ucp extension from $\osr$ to $\cstar(\osr)$ into $\B(\H)$
 \cite[Lemma 5.5]{kavruk2014}.  Furthermore,
because the extension is a homomorphism, the extension is necessarily $\pi$ and the range of the extension is
an irreducible C$^*$-subalgebra of $\B$. Hence, $\pi=\Theta+\Omega$. Because $\pi$ is irreducible, the Radon-Nikod\'ym
Theorem for completely positive linear maps \cite[Theorem 1.4.2]{arveson1969} 
states that $\Theta=s\pi$ and $\Omega=(1-s)\pi$,
for some $s\in[0,1]$. Hence, $\theta=s\phi$ and $\omega=(1-s)\phi$, thereby proving that $\phi$ is pure in $\mathcal C\mathcal P(\osr,\B)$.
\end{proof}

The following result applies in cases where $\mathfrak u$ need not necessarily be a finite set.

\begin{theorem}\label{primitive} If a discrete group $\lgG$ is generated by a set $\mathfrak u$,
and if $\cstar(\lgG)$ is a primitive C$^*$-algebra, then the embedding $\iota_e:\mathcal S(\mathfrak u)\rightarrow\cstar(\lgG)$
is pure.
\end{theorem}

 \begin{proof} By assumption of the primitivity of $\cstar(\lgG)$, there exists a faithful irreducible representation
of $\cstar(\lgG)$ on a separable Hilbert space $\H$. Thus, without loss of generality, we assume that $\cstar(\lgG)$ is a unital
C$^*$-subalgebra of $\B(\H)$, that $\mathcal S(\mathfrak u)$ is an operator subsystem of
$\cstar(\lgG)$, and that $\iota_e$ is the inclusion map $\iota_e(x)=x$, for $x\in \mathcal S(\mathfrak u)$.
Clearly $\iota_e$ maps the unitary elements of $\mathcal S(\mathfrak u)$ to unitary elements of $\cstar(\lgG)$;
hence, by Lemma \ref{kav}, $\iota_e$ is a pure element of $\mathcal C\mathcal P\left(\mathcal S(\mathfrak u), \cstar(\lgG)\right)$.
\end{proof}

A necessary condition for a full group C$^*$-algebra $\cstar(\lgG)$ 
to be primitive is that $\lgG$ be an infinite conjugacy class (i.c.c.) group. Therefore, the next theorem, which applies only
to finitely-generated groups, requires somewhat less of the group C$^*$-algebra $\cstar(\lgG)$ and answers  
question (Q2) for the operator systems 
$\oss_n$ and $\mbox{\rm NC}(n)$.

\begin{theorem}\label{cstar envelope embedding} Suppose that $\lgG$ is a finitely-generated discrete group and
$\mathfrak u$ is any finite set of generators of $\lgG$. Then the following statements are equivalent:
\begin{enumerate}
\item the canonical embedding $\iota_e:\oss(\mathfrak u)\rightarrow\cstar(\lgG)$ is pure;
\item $\lgG$ is an infinite conjugacy class (i.c.c.) group.
\end{enumerate}
\end{theorem}

\begin{proof} 
To prove that (1) implies (2), we show that if $\lgG$ is not an i.c.c. group, then 
the embedding $\iota_e:\oss(\mathfrak u)\rightarrow\cstar(\lgG)$ is not pure. Because full group algebras have trivial centre
only for i.c.c. groups (see, for example, \cite[Proposition 2.1]{murphy2003}), the ucp map 
$\iota_e:\oss(\mathfrak u)\rightarrow\cstar(\lgG)$ can never by pure if $G$ is not an i.c.c. group, by Proposition \ref{trivial centre}.

To prove that (2) implies (1), suppose that $\mathfrak u$ is given by
 $\mathfrak u=\{u_1,\dots,u_n\}$, for some unitaries $u_j\in\cstar(\lgG)$.
Note that the hypothesis that $\lgG$ is an i.c.c group implies that the reduced group C$^*$-algebra
$\cstar_{\rm r}(\lgG)$ is prime \cite[Proposition 2.3]{murphy2003}. Because separable prime C$^*$-algebras are primitive, 
there exists a faithful irreducible
representation $\pi:\cstar_{\rm r}(\lgG)\rightarrow\B(\H)$, for some Hilbert space $\H$. Let $\B=\pi\left(\cstar_{\rm r}(\lgG) \right)$,
an irreducible C$^*$-subalgebra of $\B(\H)$.
If $\rho$ denotes the left regular representation of the
group $\lgG$ (thereby implementing a homomorphism $\cstar(\lgG)\rightarrow\cstar_{\rm r}(\lgG)$), 
then define a unital completely positive linear map 
$\psi:\mathcal S(\mathfrak u)\rightarrow\B(\H)$ by
\[
\psi\left(\alpha 1+ \sum_{j=1}^n\alpha_j u_j\right)=\alpha\pi\left(\rho(1)\right)+\sum_{j=1}^n\alpha_j\pi\left(\rho(u_j)\right).
\]
Because $\psi(u)$ is unitary for every $u\in \mathfrak u$ and the commutant of 
$\{\psi(u)\,|\,u\in\mathfrak u\}$ in $\B(\H)$ is $1$-dimensional, $\psi$ is a pure element of
$\mathcal C\mathcal P\left(\mathcal R ,\B \right)$, by
Lemma \ref{kav}.

If there exists completely positive linear maps
$\vartheta,\omega:\mathcal S(\mathfrak u)\rightarrow\cstar(\lgG)$ such that $\vartheta+\omega=\iota_e$, then 
\[
\left(\pi\circ\rho\right) \circ\vartheta+\left(\pi\circ\rho\right)\circ\omega=\left(\pi\circ\rho\right)\circ\iota_e=\psi.
\]
Hence, the purity of $\psi$ 
in $\mathcal C\mathcal P(\mathcal S(\mathfrak u),\B)$
yields $\left(\pi\circ\rho\right)\circ\vartheta=s\psi=\left(\pi\circ\rho\right)\circ(s\iota_e)$, for some $s\in[0,1]$. 
Because $\psi$ is a linear map that sends a space of dimension $(n+1)$ onto a space of the same dimension, 
$\psi$ is injective; hence,
$\left(\pi\circ\rho\right)\circ\vartheta(x)=\left(\pi\circ\rho\right)\circ(s\iota_e)(x)$, for $x\in \mathcal S(\mathfrak u)$, only if 
$\vartheta(x)=s\iota_e(x)$. That is, $\vartheta=s\iota_e$ and $\omega=(1-s)\iota_e$,
proving that $\iota_e$ is pure.
\end{proof}

\begin{corollary}\label{oss embedding}
The embedding $\iota_e:\oss_n\rightarrow\cstar(\fn)$ is pure for every $n\geq 2$.
\end{corollary} 

\begin{proof}
For every $n\geq 2$, $\fn$ is an i.c.c. group. 
\end{proof}

\begin{corollary}\label{nc embedding} 
The embedding 
$\iota_e:\mbox{\rm NC}(n)\rightarrow\cstar(*_1^n\mathbb Z_2)$ is  pure for every $n\ge3$.
\end{corollary}

\begin{proof}
If $n\geq3$, then $*_1^n\mathbb Z_2$ contains $\ftwo$ as a subgroup (see, for example, \cite[Lemma D.2]{fritz2012}).
Hence, 
$*_1^n\mathbb Z_2$ is an i.c.c. group.
\end{proof}

With respect to $\oss_1$ and $\mbox{NC}(2)$, because $\cstar(\mathbb Z)$ is abelian and $\cstar(\mathbb Z_2*\mathbb Z_2)$
has nontrivial centre, Proposition \ref{trivial centre} shows that the embeddings of  $\oss_1$ and $\mbox{NC}(2)$ into their
C$^*$-envelopes are not pure.

\section{Identity Maps on Some Universal Operator Systems}

In this section we answer question (Q1) for the universal operator systems $\oss_n$ and $\mbox{\rm NC}(n)$, 
making crucial use of boundary representations and 
the fact that the operator system duals of theses operator systems can be represented as operator systems of matrices
(Theorem \ref{duals}).

\begin{theorem}\label{part 1} The identity map $\iota_n:\oss_n\rightarrow\oss_n$ is pure for every $n\geq1$.
\end{theorem}

\begin{proof}
By Theorem \ref{pure dual}, it is enough to prove that the dual map 
$\iota_n^d:\oss_n^d\rightarrow\oss_n^d$ (which is again an identity map) is pure.
By Theorem \ref{duals}, therefore, it is enough to verify that the identity map $\mathfrak i_n:\mathcal X_n\rightarrow\mathcal X_n$
is pure in the cone $\mathcal C\mathcal P(\mathcal X_n,\mathcal X_n)$, where
\begin{equation}\label{defn of osu}
\mathcal X_n=\left\{\bigoplus_{k=1}^n \left[\smallmatrix \alpha & \beta_k\\ \gamma_k & 
\alpha\endsmallmatrix\right] \vert \alpha, \beta_k,\gamma_k\in\mathbb C,  \, k= 1,\dots,n\right\}.
\end{equation}

The C$^*$-algebra generated by $\mathcal X_n$ is 
$\cstar(\mathcal X_n)=\displaystyle\bigoplus_{1}^{n}\M_2(\mathbb C)$. The
map $\pi_k:\cstar(\mathcal X_n)\rightarrow\M_2(\mathbb C)$ that projects $x\in\mathcal X_n$ onto its $k$-th direct summand 
is an irreducible representation of $\cstar(\mathcal X_n)$, for each $k=1,\dots,n$. 
We prove below that each irreducible representation $\pi_k$ of $\cstar(\mathcal X_n)$ is
a boundary representation of $\mathcal X_n$.

First note that, up to unitary equivalence, $\pi_1,\dots,\pi_n$ are all of the 
irreducible representations of $\cstar(\mathcal X_n)$; thus, the boundary representations for $\mathcal X_n$ are among these elements.
By Theorem \ref{bndry thm},
the norm of each $x\in\mathcal X_n$
is given by 
\begin{equation}\label{norming}
\|x\|=\max\left\{\|\pi(x)\|\,|\,\pi\mbox{ is a boundary representation of }\mathcal X_n\right\}.
\end{equation}
For each $k=1,\dots,n$ let $x_k\in\mathcal X_n$ be the matrix whose $k$-th direct summand 
is $\left[\begin{array}{cc}0&1\\0&0\end{array}\right]$ and whose other direct summands are the $2\times 2$ zero matrix.
Because $\pi_k(x_k)=\left[\begin{array}{cc}0&1\\0&0\end{array}\right]$ for each $k$ and $\phi_j(x_k)=0$ when $j\not=k$,
the only way that equation (\ref{norming}) can hold for every $k$ is if $\pi_k$ is a boundary representation of $\mathcal X_n$ for
every $k$. 

Now let $\phi_k=\pi_k{}_{\vert\mathcal X_n}$; thus,  $\phi_k$ is a pure element of $\mathcal C\mathcal P(\mathcal X_n,\M_2(\mathbb C))$.
Denote the canonical orthonormal basis vectors of $\mathbb C^2$ by $e_1,e_2$ and the canonical orthonormal 
basis vectors of $\mathbb C^{2n}$ by $f_\ell$, for $\ell=1,\dots 4n$.
Denote by $v_k:\mathbb C^2\rightarrow\mathbb C^{2n}$ 
the isometry that maps $e_1$ to $f_{2k-1}$ and $e_2$ to $f_{2k}$, for $k=1,\dots,n$. Thus, the ucp maps $\phi_k$ are given by
$v_k^*\mathfrak i_n v_k$.

Suppose now that $\vartheta,\omega:\mathcal X_n\rightarrow\mathcal X_n$ are completely positive linear maps for which 
$\mathfrak i_n=\vartheta+\omega$. Then
\[
\phi_{k}=v_{k}^*\mathfrak i_n v_{k} = v_{k}^*\vartheta v_{k}+v_{k}^*\omega v_{k},
\]
where $v_{k}^*\vartheta v_{k},v_{k}^*\omega v_{k}\in
\mathcal C\mathcal P(\mathcal X_n,\M_2(\mathbb C))$.
Observe that the projections $p_k=v_kv_k^*\in\M_2(\mathbb C)$, for $k=1,\dots,n$,
are mutually orthogonal, sum to the identity, and commute with
every element of $\mathcal X_n$. Thus,
\begin{equation}\label{sum eq1}
\vartheta=\sum_{k=1}^n p_k\vartheta p_k \mbox{ and } \omega=\sum_{k=1}^n p_k\vartheta p_k,
\end{equation}
and
\begin{equation}\label{sum eq2}
\mathfrak i_n=\sum_{k=1}^n p_k \mathfrak i_n p_k = \sum_{k=1}^n p_k \vartheta p_k
+ \sum_{k=1}^n p_k \omega p_k=\vartheta+\omega.
\end{equation}
Furthermore,
because each $\phi_{k}$ is a pure element of 
$\mathcal C\mathcal P(\mathcal X_n,\M_2(\mathbb C))$,
there are $s_{k} \in[0,1]$ such that $v_{k}^*\vartheta v_{k}=s_{k}\phi_{k}$
and 
$v_{k}^*\omega_{k}v_{k}=(1-s_k)\phi_{k}$ for $k=1,\dots, n$. Thus, 
\[
p_k\vartheta p_k= v_k(s_{k}\phi_{k})v_k^* = s_k p_k\mathfrak i_n p_k
\]
and
\[
p_k\omega p_k= v_k((1-s_{k})\phi_{k})v_k^* =(1- s_k) p_k\mathfrak i_n p_k
\]
for every $k$.
Therefore, equations (\ref{sum eq1}) become 
\[
\vartheta= \sum_{k=1}^n s_k (p_k\mathfrak i_n p_k) \mbox{ and }  \omega=\sum_{k=1}^n(1- s_k) p_k\mathfrak i_n p_k.
\]
Evaluation of the expression above for $\vartheta$ at the identity $1\in\M_2(\mathbb C)$ yields
 \[
 \vartheta(1 )=\sum_{k=1}^n s_k (p_k\mathfrak i_n(1) p_k)=\bigoplus_1^n  \left[\begin{array}{cc} s_k&0\\0&s_k \end{array}\right].
 \]
 Because $\vartheta(1)\in\mathcal X_n$, the diagonal of $\vartheta(1)$ is constant; hence, $s_1=\dots=s_n=s$ for some
 $s\in[0,1]$.
 That is, $\vartheta=s\mathfrak i_n$ and $\omega=(1-s)\mathfrak i_n$.
 \end{proof}

\begin{theorem}\label{part 2} 
The identity map $\iota_n:\mbox{\rm NC}(n)\rightarrow\mbox{\rm NC}(n)$ is pure for every $n\geq2$.
\end{theorem} 

\begin{proof} By Theorem \ref{pure dual}, it is sufficient to confirm that the dual map 
$\iota_n^d:\mbox{\rm NC}(n)^d\rightarrow\mbox{\rm NC}(n)^d$ (which is again an identity map) is pure.
By Theorem \ref{duals}, therefore, it is enough to verify that the identity map $\mathfrak j_n:\mathcal Y_n\rightarrow\mathcal Y_n$
is pure in the cone $\mathcal C\mathcal P(\mathcal Y_n,\mathcal Y_n)$, where
\begin{equation}\label{defn of nc dual}
\mathcal Y_n=\left\{\bigoplus_{k=1}^n \left[\smallmatrix \alpha & \beta_k\\ \beta_k & 
\alpha\endsmallmatrix\right] \vert \alpha, \beta_k\in\mathbb C,  \, k= 1,\dots,n\right\}.
\end{equation}
Note that an element $y=\displaystyle\bigoplus_{k=1}^n \left[\smallmatrix \alpha & \beta_k\\ \beta_k & 
\alpha\endsmallmatrix\right]$ is positive if and only $\beta_k\in\mathbb R$ and $\alpha\geq|\beta_k|$,
for every $k=1,\dots,n$.

The C$^*$-algebra $\mathcal Y_n$ generates
is abelian. Hence, the matrices in $\mathcal Y_n$ are commuting normal matrices and, therefore, admit a common spectral decomposition. That is, 
there is a unitary $u\in\M_{2n}(\mathbb C)$ such that
\[
u^*\left(\bigoplus_{k=1}^n \left[\smallmatrix \alpha & \beta_k\\ \beta_k & \alpha\endsmallmatrix\right]\right)u = 
\left[\begin{array}{ccccc}
\alpha+\beta_1 &&&& \\ 
&\alpha-\beta_1&&& \\
&&\ddots && \\
&&&\alpha+\beta_{n}&\\
&&&&\alpha-\beta_{n}
\end{array}
\right].
\]
Hence, $\mathcal Y_n$ is unitarily equivalent to the operator subsystem $\mathcal Z_n$ of the unital abelian C$^*$-algebra
$\ell^\infty(2n)$ given by
\[
\mathcal Z_n=\left\{ (\alpha+\beta_1,\alpha-\beta_1,\dots,\alpha+\beta_n,\alpha-\beta_n)\,\vert\,\alpha,\beta_k\in\mathbb C,\,k=1,\dots,n\right\}.
\]
Observe that $\ell^\infty(2n)$ is the unital C$^*$-algebra generated by $\mathcal Z_n$.
Let $\mathfrak z_n:\mathcal Z_n\rightarrow\mathcal Z_n$ denote the identity map on $\mathcal Z_n$. We aim to prove that $\mathfrak z_n$ is pure.

For each $k\in\{1,\dots,2n\}$ let $\pi_k:\ell^\infty(2n)\rightarrow\mathbb C$ denote the projection map onto the $k$-th coordinate. Thus, each
$\pi_k$ is an irreducible representation of $\ell^\infty(2n)$ on $\mathbb C$, and $\mathfrak R= \{\pi_1,\dots,\pi_{2n}\}$ is the set of all irreducible representations
of  $\ell^\infty(2n)$. If, for a fixed $k$, the map $\pi_k$ can be shown to be a boundary representation for $\mathcal Z_n$, then it will follow that
 $\varphi_k=\pi_k{}_{\vert\mathcal Z_n}$, the projection of $\mathcal Z_n$ onto the $k$-th coordinate, is a pure state on $\mathcal Z_n$.
 
 To this end, let $\Phi_k:\ell^\infty(2n)\rightarrow\mathbb C$ be any state extending $\varphi_k$. Because the states on $\ell^\infty(2n)$ are convex combinations
 of extremal states---which in this case are the irreducible representations $\pi_1,\dots,\pi_{2n}$---we deduce that 
 \begin{equation}\label{cnvx hull}
 \Phi_k=\sum_{j=1}^{2n}\lambda_j\pi_j,
 \end{equation}
 for some $\lambda_1,\dots,\lambda_{2n}\in[0,1]$ for which $\displaystyle\sum_{j=1}^{2n}\lambda_j=1$.
 The representation in equation (\ref{cnvx hull}) above depends on the choice of $k$ (that is, the convex coefficients $\lambda_j$ depend on the
 choice of $k$). 
 
 For notational simplicity, we consider the case of $k=1$ first. Thus, we aim to show in equation (\ref{cnvx hull})---assuming $k=1$---that $\lambda_1=1$
 and $\lambda_j=0$ for all $j\not=1$. To this end, consider the element $x\in\mathcal Z_n$ that is given by
 \[
 x=(1,0,1,0,\dots,1,0).
 \]
(We achieve $x$ by selecting $\alpha=\beta_j=\frac{1}{2}$ for every $j=1,\dots,n$.)
Thus, equation  (\ref{cnvx hull}) yields
\[
1=\varphi_1(x)=\Phi_1(x)=\sum_{j=1}^{2n}\lambda_j\pi_j(x)=\sum_{\ell=1}^{n}\lambda_{2\ell-1}.
\]
Thus, from $1=\displaystyle\sum_{j=1}^{2n}\lambda_j=\displaystyle\sum_{\ell=1}^{n}\lambda_{2\ell-1}$ we deduce that $\lambda_{2\ell}=0$ for
every $\ell=1,\dots,n$. Now using 
\[
y=(1,0,0,1,0,1,\dots,0,1)\in\mathcal Z_n,
\]
which is achieved by using $\alpha=\beta_1=\frac{1}{2}$ and $\beta_j=\frac{-1}{2}$ for $j=2,\dots,n$, we obtain
\[
1=\varphi_1(y)=\Phi_1(y)=\displaystyle\sum_{j=1}^{2n}\lambda_{j}\pi_{j}(y)=\lambda_1+\displaystyle\sum_{\ell=1}^{n}\lambda_{2\ell}=\lambda_1+0=\lambda_1.
\]
Hence, $\lambda_j=0$ for all $j\not=1$ and
$\Phi_1=\pi_1$. The arguments for every other $\Phi_k$ are handled similarly, juxtaposed according to whether $k$ is even or odd. 
Thus, for each $k=1,\dots,2n$, $\pi_k{}_{\vert\mathcal Z_n}$ has the unique extension property, and so 
 $\pi_k$ is a boundary representation for $\mathcal Z_n$. Hence, $\varphi_k=\pi_k{}_{\vert\mathcal Z_n}$ is a pure state on $\mathcal Z_n$.
 
 Now suppose that there are completely positive linear maps $\vartheta,\omega:\mathcal Z_n\rightarrow\mathcal Z_n$ for which
 $\mathfrak z_n=\vartheta+\omega$. For each $k$, let $\varphi_k=\pi_k\circ\mathfrak z_n$, $\vartheta_k=\pi_k\circ\vartheta$, and $\omega_k=\pi_k\circ\omega$, 
 which project onto the $k$-th coordinates of $x$, $\vartheta(x)$, and $\omega(x)$, respectively, for every $x\in\mathcal Z_n$. The state $\varphi_k$ was shown
in the previous paragraph to be pure, and so $\varphi_k=\vartheta_k+\omega_k$ implies that there $\vartheta_k$ and $\omega_k$ are nonnegative scalar multiples of
$\varphi_k$.

Select 
\[
x=\left(\alpha+\beta_1,\alpha-\beta_1,\dots,\alpha+\beta_n,\alpha-\beta_n \right) \in\mathcal Z_n.
\]
By the previous paragraph, each $\vartheta_k$ is nonnegative scalar multiple of
$\varphi_k$. Hence, there are nonnegative $s_1,\dots,s_n,t_1,\dots,t_n\in\mathbb R$ such that
\[
\vartheta(x)=\left(s_1(\alpha+\beta_1),t_1(\alpha-\beta_1),\dots,s_n(\alpha+\beta_n),t_n(\alpha-\beta_n) \right).
\]
However, since $\vartheta(x)\in\mathcal Z_n$, there exist
$\lambda,\mu_1\dots,\mu_n\in\mathbb C$ such that
\[
\vartheta(x)=\left(\lambda+\mu_1,\lambda-\mu_1,\dots,\lambda+\mu_n,\lambda-\mu_n \right).
\]
Hence, we have the following system of $2n$ equations:
\[
\begin{array}{rcl}
s_1(\alpha+\beta_1) &=& \lambda+\mu_1 \\
t_1(\alpha-\beta_1) &=& \lambda-\mu_1 \\
s_2(\alpha+\beta_2) &=& \lambda+\mu_2 \\
t_2(\alpha-\beta_2) &=& \lambda-\mu_2 \\
\vdots \qquad&\vdots&\quad\vdots \\
s_n(\alpha+\beta_n) &=& \lambda+\mu_n \\
t_n(\alpha-\beta_n) &=& \lambda-\mu_n \,.
\end{array}
\]
Viewing the equations above as $n$ pairs of equations, adding the first two equations in each pair leads to:
\begin{equation}\label{eqs}
(s_k+t_k)\alpha + (s_k-t_k)\beta_k=2\lambda,\mbox{ for every }k=1,\dots,n.
\end{equation}
That is, 
\[
(s_k+t_k)\alpha + (s_k-t_k)\beta_k= (s_j+t_j)\alpha + (s_j-t_j)\beta_j, \mbox{ for every }k,j=1,\dots,n.
\]
These equations above hold regardless of the choice of $\alpha,\beta_1,\dots,\beta_n$, and so it must be that $s_k=t_k$ for each $k$.
Therefore, the $n$ equations given in (\ref{eqs}) simplify to 
\begin{equation}\label{eqs2}
2s_k \alpha =2\lambda,\mbox{ for every }k=1,\dots,n.
\end{equation}
Hence, $s_1=\dots=s_n$. If $s\in\mathbb R$ denotes this nonnegative real number, then $\vartheta(x)=sx$, for every $x\in\mathcal Z_n$.
Thus, $\vartheta=s\mathfrak z_n$ and $\omega=(1-s)\mathfrak z_n$. 

Hence, $\mathfrak z_n$ is pure, implying that $\mathfrak j_n$ and $\iota_n^d$ are pure.
 \end{proof}

We now give another example among many known examples showing 
that the restriction of a pure map to an operator subsystem need not be pure. 
What makes the example here of interest is that the restriction is made to an operator subsystem
of just $1$ dimension lower.

\begin{proposition}\label{restriction proof}
There exist operator systems $\osr$, $\ost$, and $\osq$, and a unital completely positive linear map
$\Phi:\ost\rightarrow\osq$ such that
\begin{enumerate}
\item $\osr$ is an operator subsystem of $\ost$,
\item the linear dimension of the vector space $\ost/\osr$ is $1$,
\item the map $\Phi$ is pure in $\mathcal C\mathcal P(\ost,\osq)$, and
\item the restriction of $\Phi$ to $\osr$ is not pure in $\mathcal C\mathcal P(\osr,\osq)$.
\end{enumerate}
\end{proposition}

\begin{proof} Let $\ost=\osq=\mathcal X_1$; that is,
\[
\ost=\osq=\left\{\left[\begin{array}{cc} \alpha & \beta_1\\ \beta_2 & \alpha\end{array}\right]\,\vert\,
\alpha,\beta_1,\beta_2\in\mathbb C\right\}.
\]
Let
\[
\osr=\left\{\left[\begin{array}{cc} \alpha & \beta \\ \beta  & \alpha\end{array}\right]\,\vert\,
\alpha,\beta \in\mathbb C\right\},
\]
which is an operator subsystem of $\ost$ of co-dimension $1$.

Let $\Phi:\ost\rightarrow\osq$ be the identity map. By (the proof of) Theorem \ref{part 1}, $\Phi$ is a pure element of
$\mathcal C\mathcal P(\ost,\osq)$.
The completely positive linear maps $\vartheta,\omega:\ost\rightarrow\osq$ given by
\[
\vartheta\left(\left[\begin{array}{cc} \alpha & \beta_1\\ \beta_2 & \alpha\end{array}\right] \right)
=\left[\begin{array}{cc} \frac{\alpha}{2} + \frac{\beta_1+\beta_2}{4} &\frac{\alpha}{2} + \frac{\beta_1+\beta_2}{4}\\ 
\frac{\alpha}{2} + \frac{\beta_1+\beta_2}{4} & \frac{\alpha}{2} + \frac{\beta_1+\beta_2}{4}\end{array}\right]
\]
and 
\[
\omega\left(\left[\begin{array}{cc} \alpha & \beta_1\\ \beta_2 & \alpha\end{array}\right] \right)
=
\left[\begin{array}{cc} \frac{\alpha}{2} - \left(\frac{\beta_1+\beta_2}{4}\right) & \frac{-\alpha}{2} + \frac{\beta_1+\beta_2}{4} \\ 
\frac{-\alpha}{2} + \frac{\beta_1+\beta_2}{4}   &\frac{\alpha}{2} - \left(\frac{\beta_1+\beta_2}{4}\right)\end{array}\right]
\]
satisfy, when $\beta_1=\beta_2=\beta$, 
\[
\vartheta\left( \left[\begin{array}{cc} \alpha & \beta \\ \beta  & \alpha\end{array}\right]\right) 
+
\omega\left( \left[\begin{array}{cc} \alpha & \beta \\ \beta  & \alpha\end{array}\right]\right)
=
\left[\begin{array}{cc} \alpha & \beta \\ \beta  & \alpha\end{array}\right].
\]
That is, $\vartheta(x)+\omega(x)=\Phi_{\vert\osr}(x)$, for every $x\in \osr$. Hence, as neither $\vartheta$ nor $\omega$ is a scalar multiple of $\phi_{\vert\osr}$,
the restriction of $\Phi$ to $\osr$ is not pure in the cone
$\mathcal C\mathcal P(\osr,\osq)$.
\end{proof}

If $\osr$ is any operator system, then $\osr\omin\mathcal X_n$ is canonically an operator subsystem of $\osr\omin\M_n(\mathbb C)=\M_n(\osr)$; thus,
 we may describe $\osr\omin\mathcal X_n$ as the operator system
 \[
 \osr\omin\mathcal X_n=\left\{ \bigoplus_{k=1}^n\left[\begin{array}{cc} r & a_k\\ b_k & r\end{array}\right] \,\vert\, r,a_k,b_k\in\osr
 \right\}.
 \]
 Likewise,
  \[
 \osr\omin\mathcal Y_n=\left\{ \bigoplus_{k=1}^n\left[\begin{array}{cc} r & c_k\\ c_k & r\end{array}\right] \,\vert\, r,c_k\in\osr
 \right\}.
 \]

As a final consequence of Theorems \ref{part 1} and \ref{part 2} and Corollary \ref{max ent pure}, we have the following application.

\begin{proposition} If $\{u_1,\dots,u_n\}$ is the generating set of unitaries for $\cstar(\fn)$, for $n\geq 1$, and 
if $\{v_1,\dots,v_m\}$ is the generating set of selfadjoint unitaries for 
$\cstar(*{}_1^m\mathbb Z_2)$, where $m\geq2$,
then
\[
\xi=\bigoplus_{k=1}^n\left[\begin{array}{cc} 1 & u_k\\ u_k^* & 1\end{array}\right] 
\]
is a pure element of the cone $\left(\oss_n\omin\mathcal X_n\right)_+$, and 
\[
\xi'=\bigoplus_{j=1}^m\left[\begin{array}{cc} 1 & v_j\\ v_j & 1\end{array}\right] 
\] 
is a pure element of the cone $\left(\mbox{\rm NC}(m)\omin\mathcal Y_m\right)_+$.
\end{proposition} 

\begin{proof} With respect to the canonical basis 
$\{1,u_1,u_1^*,\dots,u_n,u_n^*\}$ of $\oss_n$ and the canonical basis $\{1,v_1,v_2,\dots,v_m\}$ of $\mbox{\rm NC}(m)$,
the matrix representation (see \cite[Proposition 6.1]{farenick--kavruk--paulsen--todorov2014},\cite[Theorem 4.5]{farenick--paulsen2012})
of the dual basis elements of $\oss_n^d$ and $\mbox{\rm NC}(m)^d$, where $n\geq1$ and $m\geq2$, 
are given by 
\[
\{1, e_{12}^{[k]}, e_{21}^{[k]}\,|\,k=1,\dots,n\}\mbox{ and } \{1, e_{12}^{[j]}+ e_{21}^{[j]}\,|\,j=2,\dots,m\},
\]
where $e_{ij}^{[\ell]}$ denotes the matrix formed by a direct sum of $2\times 2$ matrices in which the $\ell$-th summand is the $2\times2$ matrix unit $e_{ij}$
and every other direct summand is $0$. Thus, $\xi$ and $\xi'$ represent the maximally entangled elements of $\oss_n\otimes\oss_n^d$ and $\mbox{\rm NC}(m)\otimes \mbox{\rm NC}(m)^d$,
respectively. By Corollary \ref{max ent pure}, the purity of $\xi$ and $\xi'$ follows from the purity of the identity maps on $\oss_n$, for $n\geq1$, and $\mbox{\rm NC}(m)$, for $m\geq 2$.
\end{proof}
 
 \section*{Acknowledgement}
 
 We wish to thank Rapha\"el Clou\^atre for pointing out some errors and gaps
 in our original versions of Theorems \ref{part 1} and \ref{part 2}, and for suggesting the elegant use of references \cite{arveson2008,davidson--kennedy2015}
 in the proof of Theorem \ref{part 1}. 
 

\end{document}